\documentclass[a4paper,11pt]{amsart}
\usepackage{rotating}
\usepackage{pdflscape}
\usepackage[a4paper, left=2cm, right=2cm, top=2cm,bottom=2cm]{geometry}
\usepackage{amsmath}
\usepackage{amssymb}
\usepackage{MnSymbol}
\usepackage{color}
\usepackage{enumerate}
\usepackage{amsthm}
\usepackage{listings}
\usepackage[all]{xy}
\usepackage{longtable}
\usepackage{graphicx}
\usepackage{hyperref}
\usepackage{mathrsfs}
\usepackage[style=alphabetic,backend=bibtex]{biblatex}
\theoremstyle{definition}
\newtheorem{theorem}{Theorem}[section]

\newtheorem{lemma}[theorem]{Lemma}

\newtheorem{example}[theorem]{Example}

\newcommand{\Tot}{\operatorname{Tot}}

\newcommand{\Kosz}{\operatorname{Kosz}}

\newcommand{\Hom}{\operatorname{Hom}}

\newcommand{\rank}{\operatorname{rank}}

\newcommand{\Spec}{\operatorname{Spec}}
\newcommand{\D}{\operatorname{d}\!}
\newcommand{\CC}{\mathbb{C}}

\newcommand{\ZZ}{\mathbb{Z}}
\newcommand{\NN}{\mathbb{N}}

\newcommand{\PP}{\mathbb{P}}
\newcommand{\OO}{\mathcal{O}}

\newcommand{\id}{\operatorname{id}}
\newcommand{\im}{\operatorname{im}}

\title{Some computational aspects of spectral sequences in \v Cech cohomology}

\author{Matthias Zach}

\address{
  Matthias Zach: 
  RPTU Kaiserslautern-Landau,
  Gottlieb-Daimler-Stra{\ss}e, Geb\"aude 48,
  67663 Kaisers\-lautern,
  Germany 
}

\begin{document}

\begin{abstract}
  Sheaf cohomology or, more generally, higher direct images 
  of coherent sheaves along proper morphisms are central to modern 
  algebraic geometry. However, the computation of these objects is 
  a non-trivial and expensive task which easily challenges the capacities 
  of modern computers. We describe an algorithm and its implementation 
  to compute a spectral sequence converging to 
  the higher direct images of a bounded complex of sheaves on 
  a product of projective spaces $\PP = \PP^{r_1} \times \dots \times \PP^{r_m}$ 
  over an arbitrary affine base $\Spec R$. We assume the ring $R$ 
  to be computable and the complex of sheaves to be represented by 
  an actual complex of (multi-)graded modules. 
\end{abstract}

\maketitle

\tableofcontents

\section{Results} 

Let $S = R[x_{i, j} : i=1,\dots,m, j=0, \dots, r_i]$ be the multigraded 
polynomial ring associated to a product of projective spaces 
$\PP = \PP^{r_1} \times \dots \times \PP^{r_m}$ over $\Spec R$ for some 
computable ring $R$. Given a right-bounded cocomplex of finitely generated, 
free, graded $S$-modules $(M^\bullet, \varphi^\bullet)$, one wishes to compute 
the direct image of the associated complex of sheaves $\widetilde M^\bullet$ 
along the projection map $\pi \colon \PP \to \Spec R$ 
\[
  R\pi_* \left(\widetilde M^\bullet\right)
\]
as a complex of $R$-modules up to quasi-isomorphism. 
The spectral sequence 
\[
  E_1^{p, q} = \check H^q(\widetilde M^p; \mathfrak U) \Rightarrow 
  R^{p+q}\pi_* \left(\widetilde M^\bullet\right)
\]
for the \v Cech double complex 
$\check C^\bullet(\widetilde M^\bullet, \mathfrak U)$ for an affine 
covering $\mathfrak U$ of $\PP$ converges to the cohomology 
of this direct image and can thus be considered as a reasonable approximation of it;
see Section \ref{sec:ReviewOfCechCohomology}.
After reviewing a general result for the reduction of the problem to a computation 
on finitely generated $S$- and $R$-modules in Theorem \ref{thm:Truncations}, 
we present an algorithm 
in Section \ref{sec:ApproximationViaSpectralSequences}
which systematically exploits natural massive redundance 
in the computation of the $\infty$-page of 
that spectral sequence and reduces it generically to smaller, unique
problems on finitely generated $R$-modules. We discuss the implementation 
in more detail in Section \ref{sec:TheAlgorithm}. A prototype 
will soon become available in the computer algebra system Oscar \cite{Oscar}.

While algorithms for sheaf cohomology using Tate resolutions are generally 
believed to be faster \cite{EisenbudSchreyer08}, \cite{BrownErman24}, we 
do believe that the approach via \v Cech cohomology has its merits; not least 
because of its geometric accessibility for further structure like functoriality, 
cap/cup products, and -- in the case of K\"ahler differentials -- as a recipient 
for characteristic classes and the description of Hodge structures.

\section{A brief review of \v Cech cohomology}
\label{sec:ReviewOfCechCohomology}

Let $R$ be a Noetherian, commutative ring with unit which is \textit{computable} 
(see e.g. \cite{BarakatLangeHegermann11} and \cite{CoquandMoertbergSiles12})
and 
$\PP = \PP^{r_1} \times \dots \times \PP^{r_m} \overset{\pi}{\longrightarrow} \Spec R$ 
a product 
of projective spaces over $\Spec R$. Associated to $\PP$ we have a multigraded 
ring 
\[
  S = R[x_{1, 0}, \dots, x_{1, r_1}, x_{2, 0}, \dots, x_{m, r_m}], \quad 
  \deg(x_{i, k}) = (0, \dots, 0, \underbrace{1}_{i\textnormal{-th entry}}, 0, \dots, 0) \in \ZZ^m, k = 1,\dots, r_i.
\]
A finitely generated, graded $S$-module $M$ represents a coherent sheaf
$\mathcal F = \widetilde M$ on $\PP$. Note that for a given sheaf
$\mathcal F$, there can be multiple, non-isomorphic graded modules
$M$ representing it; in particular, such a representation involves a
choice. A morphism of graded modules $\varphi \colon M \to N$
induces a morphism of coherent sheaves
\[
  \tilde \varphi : \widetilde M \to \widetilde N.
\]
This \textit{sheafification} is an exact functor and 
a (bounded) complex of $S$-modules $(M^\bullet, \varphi^\bullet)$ gives rise 
to a (bounded) complex of coherent sheaves. 

An accessible definition of the direct image of such a complex 
along the projection map $\pi \colon \PP \to \Spec R$ can be given as follows. 
Suppose $(M^\bullet, \varphi^\bullet)$ is a right-bounded cocomplex of finitely generated 
graded $S$-modules. 
Choose any covering $\mathfrak U = \{U_j\}_{j = 1}^N$ of $\PP$ by \textit{principal open subsets} 
$U_j = D(h_j) := \{h_j \neq 0\}$, $h \in S$;  
then the direct image along $\pi$
is (quasi-) isomorphic to the right-bounded complex of $R$-modules given by 
\[
  R\pi_* (\widetilde M^\bullet) \cong 
  \Tot\left(\check C^{\bullet}(\widetilde M^\bullet; \mathfrak U)\right).
\]
See e.g. the 
\href{https://stacks.math.columbia.edu/tag/01FP}{Stacks project} 
\cite{stacks-project}, Section 20.25\footnote{Retrieved on June 2nd, 2025} 
for an overview on this topic.
Here, we denote by $\check C^{\bullet}(\widetilde M^\bullet; \mathfrak U)$
the \textit{\v Cech double complex} with entries
\[
  \check C^{p, q} = \check C^q(\widetilde M^{p}; \mathfrak U) = \bigoplus_{0<j_1 < \dots < j_{q+1} \leq N} \widetilde M^{p}\left(U_{j_1} \cap \dots \cap U_{j_{q+1}}\right)
\]
where the \textit{horizontal differential} is induced by $\varphi$ on the respective summands, and 
the \textit{vertical differential} is given by the \v Cech map 
\[
  \check \partial \colon 
  \bigoplus_{0<i_1 < \dots < i_{q} \leq N} a_{i_1,\dots, i_q} \mapsto 
  \bigoplus_{0<j_1 < \dots < j_{q+1} \leq N} b_{j_1,\dots, j_{q+1}}, \quad 
  b_{j_1,\dots, j_{q+1}} = \sum_{k=1}^{q+1} (-1)^{k+1} a_{j_1,\dots, \hat{j_k}, \dots, j_{q+1}}|U_{j_1,\dots,j_{q+1}}.
\] 

By construction of the sheafification functor $\tilde \cdot \colon M \mapsto \mathcal F = \widetilde M$, 
we obtain the sections of $\mathcal F$ on an affine chart $U = D(f) \subset \PP$ as
\[
  \mathcal F(U) \cong \left(M[f^{-1}]\right)_0,
\]
the \textit{degree-$0$-strand} of the localization of $M$ at $f$. 
It is a common exercise in commutative algebra to show that $M[f^{-1}]$ is 
actually isomorphic to the limit of the directed system 
\[
  M \overset{f\cdot -}{\longrightarrow} 
  M \overset{f\cdot -}{\longrightarrow} 
  M \overset{f\cdot -}{\longrightarrow} 
  M \overset{f\cdot -}{\longrightarrow} 
  \cdots
\]
where we can think of the maps as extending a fraction 
\[
  \frac{a}{f^k} \mapsto \frac{f\cdot a}{f^{k+1}}, \quad a \in M.
\]
In this sense, the \v Cech complex on the sheaf associated to a graded module $M$
can be seen as a direct limit of 
strands\footnote{For this note, we will follow the convention that the $\Hom(-, M)$-functor 
preserves the direction of maps, but 
flips the sign of the cohomological degrees in its first argument, i.e. 
$C_k$ in cohomological degree $k$ is taken to $\Hom(C_k, M)$ in \textit{co}homological 
degree $-k$.}
\begin{equation}
  \label{eqn:CechComplexAsDirectLimit}
  \check C^\bullet(\widetilde M; \mathfrak U) \cong 
  \lim_{k\to \infty} \Hom\left(\Kosz^*(h_1^k,\dots, h_N^k), M\right)_0
\end{equation}
where we denote by 
\begin{equation}
  \label{eqn:DefTruncatedKoszulComplex}
  \Kosz^*(f_1,\dots, f_n) : 
  0 \longrightarrow
  \bigwedge^{n} F \overset{f\invneg}{\longrightarrow} 
  \bigwedge^{n-1} F \overset{f\invneg}{\longrightarrow} 
  \cdots \overset{f\invneg}{\longrightarrow} 
  \bigwedge^{1} F \overset{f\invneg}{\longrightarrow} 
  0, 
  \quad
  F = \bigoplus_{i=1}^n S(-\deg f_i).
\end{equation}
the truncated \textit{graded Koszul cocomplex} 
for the contraction with $f = (f_1,\dots, f_n)$
\[
  f\invneg \, \colon 
  e_{i_1} \wedge \dots \wedge e_{i_p} \mapsto \sum_{k=1}^p (-1)^{i+1} f_{i_k} \cdot 
  e_{i_1} \wedge \dots \wedge \widehat{e_{i_k}} \wedge \dots \wedge e_{i_p},
\]
but shifted so that the term $\bigwedge^1 F$ appears in cohomological degree zero
and the other non-trivial ones in negative degrees.
The functoriality for the Koszul cocomplex (see e.g. \cite{Zach24}) then 
provides us with the structure of the direct limit. 

For a fixed value of $k$, an element in cohomological degree $p$ on the right hand side of 
(\ref{eqn:CechComplexAsDirectLimit}), say 
\[
  e_{i_0} \wedge \dots \wedge e_{i_p} \mapsto a \in M,
\]
then stands for the fraction $\frac{a}{h_{i_0}^k \cdots h_{i_p}^k}$ in the respective 
direct summand of the \v Cech complex. 

From these explicit descriptions we may now conclude the following theorem 
as a first step to rendering the approach to the direct image via \v Cech 
cohomology computable. We refer to \cite[Section 2]{Hartshorne67} 
and the references therein for a more detailed account.

\begin{theorem}
  \label{thm:CechCohomologyAsDirectLimit}
  Let $(M^\bullet, \varphi^\bullet)$ be a right-bounded cocomplex of finitely 
  generated $S$-modules. Then the direct image of the associated 
  complex of sheaves $\widetilde M^\bullet$ on $\PP$ along the projection 
  map $\pi \colon \PP \to \Spec R$ is given by 
  \begin{equation}
    \label{eqn:PushforwardAsDirectLimit}
    R\pi_* \widetilde M^\bullet = 
    \lim_{k \to \infty}
    \Tot\left(\Hom\left(\Kosz^*(h_1^k,\dots,h_N^k), M^\bullet\right)\right)_0.
  \end{equation}
\end{theorem}

\section{Approximation by finitely generated submodules}
\label{sec:ApproximationByFiniteSubmodules}

While Equation (\ref{eqn:PushforwardAsDirectLimit}) in 
Theorem \ref{thm:CechCohomologyAsDirectLimit} 
is fairly straightforward to write down, 
it contains some technical subtleties which obstruct accessibility for practical 
computations. First of all the fact that the involved $R$- and $S$-modules are not 
anymore finitely presented due to the limit construction. 
When it comes to computing cohomology of $R\pi_* \widetilde M^\bullet$ 
or quasi-isomorphic complexes up to a specific cohomological degree, we may, 
however, pass to suitable truncations in the pole orders $k$. This will 
be detailed in Theorem \ref{thm:Truncations} below and we will now start 
preparing for its proof.

\medskip

A natural first step in the computation of $R\pi_* \widetilde M^\bullet$ 
is to replace $M^\bullet$ by some 
(Cartan-Eilenberg) resolution of free $S$-modules.
We will henceforth and without loss 
of generality assume that $M^\bullet$ is of the form
\begin{equation}
  \label{eqn:CartanEilenbergResolution}
  M^\bullet \, :\,
  0 \longrightarrow
  \bigoplus_{i_l=1}^{b_{{p_0}-l}} S(-d_{p_0-l, i_l})
  \overset{A_{p_0 + l}}{\longrightarrow}
  \cdots
  \overset{A_{p_0-2}}{\longrightarrow}
  \bigoplus_{i_1=1}^{b_{p_0-1}} S(-d_{p_0+1, i_1})
  \overset{A_{p_0-1}}{\longrightarrow}
  \bigoplus_{i_0=1}^{b_{p_0}} S(-d_{p_0, i_0})
  \longrightarrow
  0
\end{equation}
for some $p_0 \in \ZZ$, shifts $d_{p, i} \in \ZZ^m$, and 
morphisms $A_p$ of degree zero. If the complex is unbounded 
to the left, we will subsume this under the case $l=\infty$. 

\medskip
For products of projective spaces, there is a pretty simple (almost canonical) choice 
for the polynomials $h_i$. In case of just one factor $m=1$, we choose the 
coordinate functions
\[
  h_1,\dots, h_{r_1+1} = x_0, \dots, x_{r_1}
\]
and obtain the truncated Koszul cocomplexes in the powers of these elements 
for the expressions in Theorem \ref{thm:CechCohomologyAsDirectLimit}. 
In case of a product with more than one factor, the straightforward choice 
is to take products of the variables for the respective factors:
\[
  h_{i_1,\dots, i_m} = x_{1, i_0} \cdot \dots \cdot x_{m, i_m} 
\]
for all combinations $0 \leq i_k \leq r_k$, $k = 1,\dots, m$.
Note that the resulting truncated Koszul cocomplex turns out to be (isomorphic to) 
a tensor product of the truncated Koszul cocomplexes for the individual factors. 
In fact, the latter allows for a more fine grained control over the exponents
of the denominators which makes it the preferred cocomplex to work with. 
For a $\ZZ^m$-graded ring $S$ as above and an \textit{exponent vector} 
$e = (e_1,\dots, e_m) \in \NN_0^m$ we therefore set 
\begin{equation}
  \label{eqn:TruncatedCechCocomplex}
  K_{\leq e}^\bullet = 
  K_{\leq e}^\bullet(S) := 
  \Kosz^*\left(x_{1, 0}^{e_1},\dots, x_{1, r_1}^{e_1}\right) \otimes
  \dots
  \otimes
  \Kosz^*\left(x_{m, 0}^{e_m},\dots, x_{m, r_m}^{e_1}\right).
\end{equation}
With these choices, the morphisms of strands in the direct limit 
on the right hand side of (\ref{eqn:PushforwardAsDirectLimit})
turn into inclusions (and extensions) of monomial diagrams.
The following example should give a feeling for 
the explicit form this actually takes.

\begin{example}
  \label{exp:StrandExhaustion}
  Consider the case $\PP = \PP^1$ with homogeneous coordinate ring 
  $S = R[x, y]$ over $R$ and the strand for $-d = -4$. We illustrate 
  this example in Figure \ref{fig:MonomialDiagramsCechComplex}.
  The direct image for $\widetilde S(-4) = \OO(-4)$ 
  has homology $R^3$ in cohomological degree $1$ and zero elsewhere. 
  \begin{figure}
    \centering
    \includegraphics[page=2, clip=true, scale=0.5, trim=0 7cm 0 0]{monomial_diagrams.pdf}
    \[
      \downarrow \check \partial
    \]
    \includegraphics[page=1, clip=true, scale=0.5, trim=0 7cm 7cm 0]{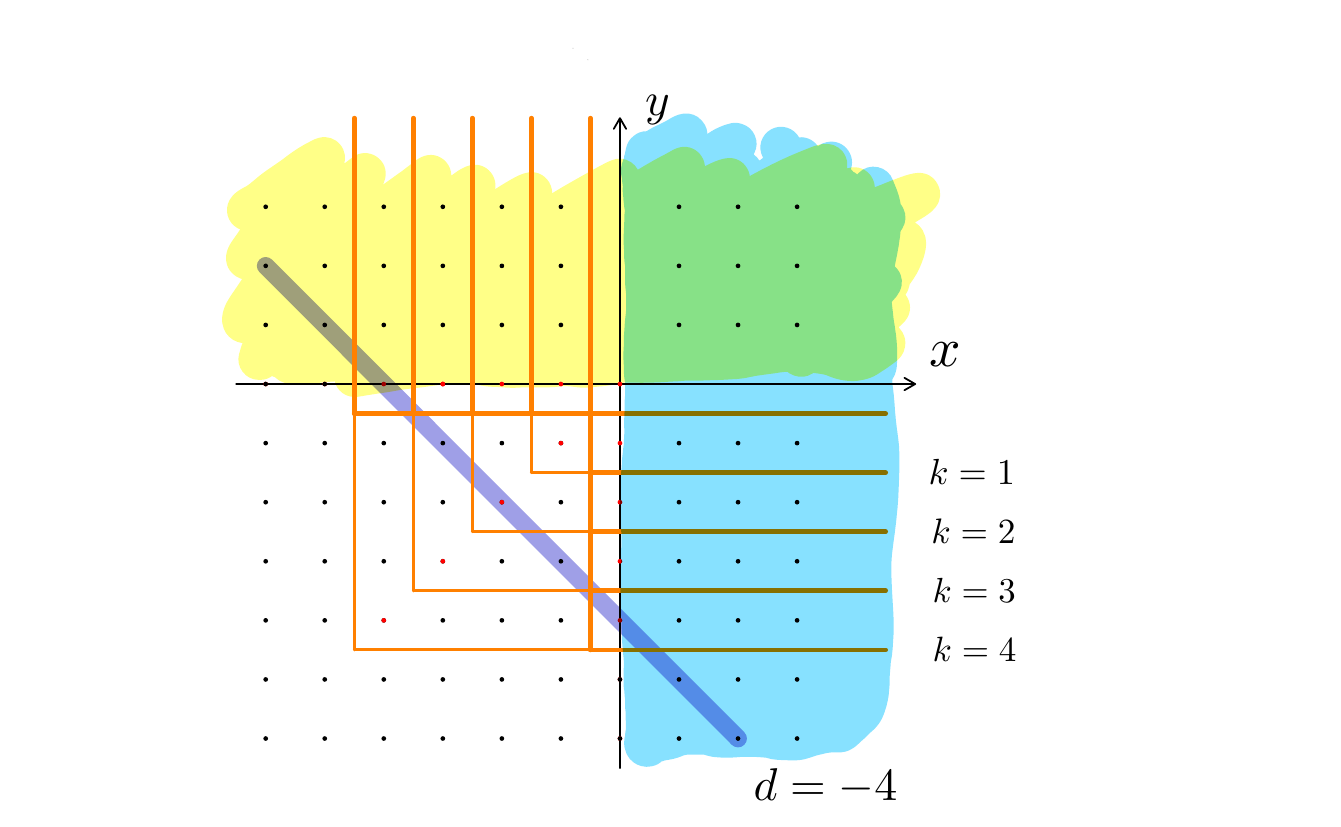}
    \label{fig:MonomialDiagramsCechComplex}
    \caption{Monomial diagrams for the exhaustion of the \v Cech complex from Example \ref{exp:StrandExhaustion} as a direct limit; the outlines of the
    quadrants for the contributions from the duals of the 
    truncated Koszul complexes for different $k$ in orange and the monomials 
    on the strand for $d=-4$ marked in dark blue}
  \end{figure}
  For the exponent $k = 1$ we have a the zero complex
  \[
    \Hom\left(\Kosz^*\left(x^1, y^1\right), S\right)_{-1} 
    \cong 
    \left(
      R^0 \oplus R^0 \overset{\varphi}{\longrightarrow} 
      R^0
    \right),
  \]
  simply, because all strands are empty in these degrees. 
  Starting from exponent $k=2$, we begin to partially recover the 
  cohomology
  \[
    \Hom\left(\Kosz^*\left(x^2, y^2\right), S\right)_{-2} 
    \cong 
    \left(
      R^0 \oplus R^0 \overset{\varphi}{\longrightarrow} 
      R^1
    \right),
  \]
  where the single non-zero generator stands for the monomial 
  $\frac{1}{x^2 y^2}$. 
  For $k=3$ we obtain a complex which is quasi-isomorphic to 
  the direct image
  \[
    \Hom\left(\Kosz^*\left(x^3, y^3\right), S\right)_{-3} 
    \cong 
    \left(
      R^0 \oplus R^0 \overset{\varphi}{\longrightarrow} 
      R^3
    \right),
  \]
  with the three generators being $\frac{x^2}{x^3y^3}, \frac{xy}{x^3 y^3}$, and $\frac{y^2}{x^3 y^3}$.
  From $k=4$ and onwards, the strands become bigger than necessary
  \[
    \Hom\left(\Kosz^*\left(x^k, y^k\right), S\right)_{-4} 
    \cong 
    \left(
      R^{k-3} \oplus R^{k-3} \overset{\varphi}{\longrightarrow} 
      R^{2k - 3}
    \right)
  \]
  and we observe cancellation of monomials when passing to the cohomology.
\end{example}

The particularly simple form of the morphisms in the direct limit as inclusions of monomial 
diagrams allows us to write down not only their inclusions for $k \leq k'$, but also 
projections
\[
  \iota_k^{k'} \colon
  \Hom\left(\Kosz^*\left(x^{k'}, y^{k'}\right), S\right)_{-d} 
  \longleftrightarrow 
  \Hom\left(\Kosz^*\left(x^k, y^k\right), S\right)_{-d}
  \colon \pi_k^{k'}
\]
so that $\pi_k^{k'}$ is a left-inverse to $\iota_k^{k'}$.
Given the context of the \v Cech 
complex, these projections are similar to ``division with remainder'', but forgetting 
eventual remainders. 

We would like to reduce the computation of the direct image to a procedure 
involving only finitely generated $S$- and $R$-modules. To this end, we 
present the following preparatory lemma.
\begin{lemma}
  \label{lem:HomotoyEquivalence}
  For any degree $d \in \ZZ^m$ there exists a \textit{minimal exponent vector}
  $e_\textnormal{min}(d)$ so that for 
  every pair $e_\textnormal{min}(d) \leq e \leq e'$ 
  and the associated inclusion and projection morphisms
  \[
    \iota \colon \Hom\left(K_{\leq e}^\bullet(S), S\right)_{-d} \longleftrightarrow 
    \Hom\left(K_{\leq e'}^\bullet(S), S\right)_{-d} \colon \pi
  \]
  the composition $\iota \circ \pi$ is homotopy equivalent to the identity 
  \[
    \mathrm{id} - \iota \circ \pi = h \circ \check \partial + \check \partial \circ h
  \]
  on the complex for $e'$ for some homotopy
  \[
    h = h_e^{e'} \colon 
    \Hom\left(K_{\leq e'}^\bullet(S), S\right)_{-d}
    \longrightarrow
    \Hom\left(K_{\leq e'}^\bullet(S), S\right)_{-d}[-1].
  \]
\end{lemma}

\noindent
Here we write $e \leq e'$ if and only if $e_i \leq e'_i$ for every $i = 1,\dots, m$. 

\begin{proof}
  We leave the case $m=1$ of one factor $\PP = \PP^n$ to the reader and refer to 
  Example \ref{exp:StrandExhaustion}
  for an illustration. The minimal exponent vector for degree $-d$ is 
  \[
    e_\textnormal{min}(-d) = 
    \begin{cases}
      0 & \textnormal{if } (-d) \geq 0\\
      d-1 & \textnormal{otherwise}.
    \end{cases}
  \]
  The general case can then be proved by induction. Suppose we have two pairs of complexes, 
  both endowed with an injection and a projection 
  \[
    \iota_i \colon C_i^\bullet \longleftrightarrow D_i^\bullet \colon \pi_i, \quad i = 1, 2
  \]
  and homotopies $h_i \colon D_i^\bullet \longrightarrow D_i^\bullet[-1]$ for $\id_{D_i} - \iota_i \circ \pi_i$.
  Consider the induced injections and projections on the tensor product
  \[
    \iota_1 \otimes \iota_2 \colon C_1^\bullet \otimes C_2^\bullet \longleftrightarrow
    D_1^\bullet \otimes D_2^\bullet
    \colon \pi_1 \otimes \pi_2
  \]
  and their total complexes, and the morphism
  \[
    H:=\mathrm{Tot}\left(h_1 \otimes \mathrm{id}_{D_2^\bullet} + \mathrm{id}_{C_1^\bullet} \otimes h_2\right)
    \colon 
    \mathrm{Tot}\left(D_1^\bullet \otimes D_2^\bullet\right)
    \longrightarrow \mathrm{Tot}\left(D_1^\bullet \otimes D_2^\bullet\right)[-1].
  \]
  Then it is easy to check that in fact 
  \[
    H \circ D + D \circ H = \mathrm{id}_{\mathrm{Tot}(D_1^\bullet\otimes D_2^\bullet)}
    - \mathrm{Tot}(\iota_1\otimes \iota_2) \circ \mathrm{Tot}(\pi_1\otimes \pi_2)
  \]
  for $D$ the differential on $\mathrm{Tot}(D_1^\bullet \otimes D_2^\bullet)$.
\end{proof}

\begin{theorem}
  \label{thm:Truncations}
  Suppose $M^\bullet$ is a \textit{bounded} complex of free, graded $S$-modules 
  as in (\ref{eqn:CartanEilenbergResolution}). 
  Let $e_0 \in \NN^m$ be the supremum\footnote{We refer to the supremum with respect to the partial ordering 
  given by the natural ordering on the components} of the minimal exponent vectors 
  $e_\textnormal{min}(-d)$ for $-d = -d_{p, i_p}$ the shifts appearing in $M^\bullet$.
  Then the direct image 
  \[
    R\pi_* \widetilde M^\bullet \cong_{\textnormal{qis}}
    \Tot\left(\Hom\left(K_{\leq e_0}^\bullet(S), M^\bullet\right)\right)_0
  \]
  is quasi-isomorphic to degree-$0$-strand of the complex of finite 
  $S$-modules on the right hand side.
\end{theorem}

\begin{proof}
  It is a well known fact that taking cohomology commutes with direct limits. 
  Therefore, Theorem \ref{thm:CechCohomologyAsDirectLimit} provides us with 
  \[
    R^k\pi_* \widetilde{M^\bullet} =
    H^k\left(
    \lim_{e \to \infty} 
    \Tot\left(\Hom\left(K_{\leq e}^\bullet(S), M^\bullet\right)\right)_0
    \right)
    = 
    \lim_{e \to \infty} 
    H^k\left(
    \Tot\left(\Hom\left(K_{\leq e}^\bullet(S), M^\bullet\right)\right)_0
    \right).
  \]
  The limit on the right hand side is taken over the induced maps in 
  cohomology for the inclusions of cocomplexes
  for $e \leq e' \in \NN_0^m$ which is actually cofinite with the 
  limit in Theorem \ref{thm:CechCohomologyAsDirectLimit}.

  Let $e' \geq e_0$ be arbitrary and consider the inclusion of total 
  complexes
  \[
    \iota \colon 
    \Tot\left(\Hom\left(K_{\leq e_0}^\bullet(S), M^\bullet\right)\right)_0
    \hookrightarrow 
    \Tot\left(\Hom\left(K_{\leq e'}^\bullet(S), M^\bullet\right)\right)_0.
  \]
  Denote by $h$ the morphism on the columns of the double complex 
  $\Hom\left(K_{\leq e'}^\bullet(S), M^\bullet\right)_0$
  induced by the $h_{e_0}^{e'}$ from Lemma \ref{lem:HomotoyEquivalence} 
  on the respective direct sums. We first show that the map on 
  cohomology induced by $\iota$ is surjective. 
  \[
    \begin{xy}
      \xymatrix{
        \cdots \ar[r]^-\varphi &
        \Hom\left(K_{\leq e'}^{q}, M^{k-q}\right)_0 
        \ar@/_1pc/[d]_h 
        \ar[r]^-\varphi & 
        \Hom\left(K_{\leq e'}^{q}, M^{k-q+1}\right)_0 
        \ar@/_1pc/[d]_h 
        &
        & & \\
        &
        \Hom\left(K_{\leq e'}^{q-1}, M^{k-q}\right)_0 
        \ar[u]_{\check \partial} \ar[r]^-\varphi & 
        \Hom\left(K_{\leq e'}^{q-1}, M^{k-q+1}\right)_0 
        \ar@/_1pc/[d]_h 
        \ar[u]_{\check \partial} \ar[r]^-\varphi & 
        \Hom\left(K_{\leq e'}^{q-1}, M^{k-q+2}\right)_0 
        \ar@/_1pc/[d]_h 
        &
        & \\
        & & 
        \Hom\left(K_{\leq e'}^{q-2}, M^{k-q+2}\right)_0 
        \ar[u]_{\check \partial} \ar[r]^-\varphi & 
        \Hom\left(K_{\leq e'}^{q-2}, M^{k-q+3}\right)_0 
        \ar[u]_{\check \partial} \ar[r]^-\varphi & 
        \cdots & & 
      }
    \end{xy}
  \]
  Let $a = (a_0, a_1, \dots) \in 
  \Hom\left(K_{\leq e'}^0, M^k\right)_0 \oplus
  \Hom\left(K_{\leq e'}^1, M^{k-1}\right)_0 \oplus \cdots
  $
  be an abitrary element in the kernel of the differential $D$ 
  on the total complex 
  $\Tot\left(\Hom\left(K_{\leq e'}^\bullet(S), M^\bullet\right)\right)_0$ 
  in cohomological degree $k$. Here we will follow the sign 
  convention that $D$ is given by 
  \[
    \check \partial \oplus (-1)^{i + j -1} \varphi \colon
    \Hom\left(K_{\leq e'}^j, M^i\right)_0 \to
    \Hom\left(K_{\leq e'}^{j+1}, M^i\right)_0 \oplus
    \Hom\left(K_{\leq e'}^j, M^{i+1}\right)_0
  \]
  on the respective summands of the total complex.
  We construct an element 
  $\overline a$ in the subcomplex 
  $\Tot\left(\Hom\left(K_{\leq e_0}^\bullet(S), M^\bullet\right)\right)_0$ 
  which represents the same cohomology class as $a$, starting 
  with the last value $a_q \in \Hom\left(K_{\leq e'}^q, M^{k -q}\right)_0$ 
  for which $K_{\leq e'}^q$ is non-zero.

  Let $b_{q-1} = h(a_q) \in \Hom\left(K_{\leq e'}^{q-1}, M^{k-q}\right)_0$.
  Then by construction of $h$, the element $\overline a_q = a_q - \check \partial b_{q-1}$
  lays in the submodule $\Hom\left(K_{\leq e_0}^q, M^{k -q}\right)_0$.
  This ``correction term'' $b_{q-1}$ for $a_q$ contributes an unavoidable 
  increment on $a_{q-1}$ in the total complex, 
  namely $(-1)^{k} \cdot \varphi(b_{q-1})$. In an attempt to 
  reproduce the previous step, we set 
  \[
    \overline a_{q-1} := 
    a_{q-1} - (-1)^k \varphi(b_{q-1}) 
    - \check \partial \circ h\left(
    a_{q-1} - (-1)^k \varphi(b_{q-1}) 
    \right)
  \]
  and claim that, again, $\overline a_{q-1}\in 
  \Hom\left(K_{\leq e_0}^{q-1}, M^{k -q+1}\right)_0 
  \subset \Hom\left(K_{\leq e'}^{q-1}, M^{k -q+1}\right)_0$
  comes to lay in the submodule for $e_0$. To see this, note that by construction 
  of $h$, this would be the case for the element 
  \begin{eqnarray*}
    & & 
    a_{q-1} - (-1)^k \varphi(b_{q-1}) 
    + \left(\check \partial \circ h
    + h\circ \check \partial\right)\left(
    a_{q-1} - (-1)^k \varphi(b_{q-1}) 
    \right)\\
    & = & 
    a_{q-1} - (-1)^k \varphi(b_{q-1}) 
    + \check \partial \circ h
    \left(
    a_{q-1} - (-1)^k \varphi(b_{q-1}) 
    \right)
    + h
    \left(
    \check \partial(a_{q-1}) - (-1)^k \check \partial(\varphi(b_{q-1}))
    \right)
    \\
    & = & 
    a_{q-1} - (-1)^k \varphi(b_{q-1}) 
    + \check \partial \circ h\left(
    a_{q-1} - (-1)^k \varphi(b_{q-1}) 
    \right)
    + (-1)^k h\left(\varphi(a_q) - \varphi(\check \partial(b_{q-1}))\right)\\
    & = & 
    a_{q-1} + (-1)^k \varphi(b_{q-1}) 
    + \check \partial \circ h\left(
    a_{q-1} + (-1)^k \varphi(b_{q-1}) 
    \right)
    + (-1)^k h\left(\varphi(\overline a_q)\right)
  \end{eqnarray*}
  since, by assumption on $a$ to be in the kernel of $D$, we have $\check \partial a_{q-1} = -(-1)^{k-1} \varphi(a_q)$.
  But as $\varphi$ restricts to the subcomplex for $e_0$, we find that with 
  $\overline a_q$ also its image
  \[
    \varphi(\overline a_q) \in 
    \Hom\left(K_{\leq e_0}^q, M^{k -q-1}\right)_0 \subset
    \Hom\left(K_{\leq e'}^q, M^{k -q-1}\right)_0
  \]
  comes to lay in the submodule for $e_0$. Then 
  $h(\varphi(\overline a_q)) = 0$ by construction of $h$ and 
  we arrive at the desired property for $\overline a_{q-1}$. 
  Setting $b_{q-2} = h(a_{q-1} - (-1)^k\varphi(b_{q-1}))$, we can repeat that 
  procedure until we reach $a_0$. The resulting cochain 
  $b = (b_0, b_1, \dots, b_{q-1}, 0)$ then relates 
  \[
    \overline a = a - D b
  \]
  as elements in the total complex which concludes the 
  desired surjectivity of $\iota^*$ on the $k$-th cohomology.

  Injectivity of $\iota^*$ on cohomology can be established in a similar 
  way. Suppose $c = Da$ for some $c$ in the subcomplex for $e_0$ at cohomological degree $k+1$, 
  with a priori no restrictions on $a$ in cohomological degree $k$.
  We need to replace $a$ by an element $\overline a$ in the image of $\iota$.
  Again, we set $b_{q-1} = h(a_q)$ so that $\overline a_q := a_q - \check \partial b_{q-1}$ 
  is in the subcomplex for $e_0$. Then we subsequently set 
  \[
    \overline a_{q-1} = 
    a_{q-1} - (-1)^k \varphi(b_{q-1}) 
    - \check \partial \circ h \left(
    a_{q-1} - (-1)^k\varphi(b_{q-1}) 
    \right).
  \]
  Commutativity of the squares assures that 
  \[
    \check\partial(\overline a_{q-1}) =
    \check \partial(a_{q-1}) - (-1)^k\check\partial (\varphi(b_{q-1}))
    = (-1)^k\left(
    \varphi(a_q) - \varphi(\check \partial(b_{q-1}))\right)
    = (-1)^k\varphi(\overline a_q)
  \]
  is in the image of $\iota$ and hence annihilated by $h$. 
  Thus, for the same reasons as before, $\overline a_{q-1}$ must 
  also be contained in the image of $\iota$ and we can proceed inductively 
  with $b_{q-2} = h(a_{q-1} - (-1)^k\varphi(b_{q-1}))$ to produce cochains 
  $b$ and $\overline a$ so that $a - \overline a = D b$ and hence 
  \[
    c = D a = D(\overline a + Db) = D \overline a
  \]
  for $\overline a$ in the image of $\iota$, as required.
\end{proof}

%

Even with the aid of Theorem \ref{thm:Truncations}, 
we are eventually facing a second challenge: 
The sheer size of the $R$-modules for the strands tends to explode with 
increasing complexity of the input $M^\bullet$. 
An example for this behavior which has actually been driving the 
work on this project, has been given in \cite[Example 5.3]{Zach24}. 
We briefly reproduce it here.

\begin{example}
  \label{exp:2x3Example}
  Consider the space of complex $2\times 3$-matrices 
  $\CC^{2\times 3}$ and the subvariety 
  \[
    X = \left\{\varphi = 
    \begin{pmatrix}
      x & y & z \\
      u & v & w
    \end{pmatrix}
    \in \CC^{2\times 3} : \rank \varphi < 2\right\},
  \]
  with its natural stratification by rank. 
  For any pair of integers $l, k > 0$, the restriction of the function 
  \[
    f \colon \CC^{2\times 3} \to \CC^2, \quad
    \begin{pmatrix}
      x & y & z \\
      u & v & w
    \end{pmatrix}
    \mapsto 
    (x-v^k, w - y^l)
  \]
  to $X$ 
  has a singularity (in the stratified sense) which is isomorphic to the 
  third entry 
  in the list of simple Cohen-Macaulay codimension $2$ surface 
  singularities, \cite[Theorem 3.3]{FruehbisKruegerNeumer10}. 

  Associated to $f$ there is a topological invariant 
  called the \textit{Milnor number}
  which turns out to be 
  \[
    \mu(1; f) = k + l - 2.
  \]
  According to the findings in \cite{Zach24}, this number must coincide 
  with following sum of holomorphic Euler characteristics
  \begin{eqnarray*}
    \mu(1; f) &=& 
    \chi\left(R(\nu_1)_* \left(\mathcal Kosz(\nu_1^*(f_1 + f_2)) \otimes \mathcal ENC(\nu_1^*\D f)\right)\right) \\
    & & -
    \chi\left(R(\nu_1)_* \left(\mathcal Kosz(\nu_1^*(f_1 + f_2)) \otimes \mathcal ENC(\nu_1^*\D (f_1 + f_2, l_2)\right)\right) \\
    & & + \chi\left(R(\nu_1)_* \left(\mathcal Kosz(\nu_1^*(l_2)) \otimes \mathcal ENC(\nu_1^*\D (f_1 + f_2, l_2))\right)\right) \\
    & & -
    \chi\left(R(\nu_1)_* \left(\mathcal Kosz(\nu_1^*(l_2)) \otimes \mathcal ENC(\nu_1^*\D (l_1, l_2)\right)\right)
  \end{eqnarray*}
  where each one of the Euler characteristics itself is equal to a non-negative 
  topological intersection number. Here, $\nu_1$
  is the projection of the \textit{Nash transformation}
  \[
    \begin{xy}
      \xymatrix{
        \widetilde X \ar@{^(->}[r] \ar[d]_{\nu_1} & 
        \PP^2 \times \PP^1 \times \CC^{2\times 3} \ar[d]\\
        X \ar@{^(->}[r] & \CC^{2\times 3}
      }
    \end{xy}
  \]
  of $X \hookrightarrow \CC^{2\times 3}$ and $l_1 = x + v$ and $l_2 = y + w$ are ``sufficiently general linear forms'' in 
  the sense of \cite{Zach24}.
  The construction of the cocomplexes of bigraded modules representing the 
  respective tensor products of cocomplexes of sheaves 
  on $\PP^2 \times \PP^1 \times \CC^{2\times 3}$
  is also explained in \cite{Zach24}.

  The promise of \cite{Zach24} was to render the Milnor numbers 
  computable via symbolic computations through the above formula. 
  However, computations 
  with the naive approach from Theorem \ref{thm:Truncations}
  using {\v C}ech cohomology 
  for the first one of the above four summands, 
  produces a cocomplex of free $\CC[x, y, z, u, v, w]$-modules with 
  the ranks listed in Table \ref{tab:BettiNumbersPushforward}. 
  The modules in this raw output are much too big to be of any 
  practical use at the moment.
  \begin{table}
    \label{tab:BettiNumbersPushforward}
    \begin{center}
      \begin{tabular}{|r|r|r|r|r|r|r|r|}
        \hline
        cohom. degree & 6 & 5 & 4 & 3 & 2 \\
        \hline
        rank & 0 & 8,125 & 202,531 & 2,086,354 & 12,040,471 \\
        \hline
        \hline
        cohom. degree & $1$ & $0$ & $-1$ & $-2$ & $-3$ \\
        \hline
        rank & 45,096,331 & 119,557,797 & 237,459,216 & 367,251,441 & 453,451,541 \\ 
        \hline
        \hline
        cohom. degree & $-4$ & $-5$ & $-6$ & $-7$ & $-8$ \\ 
        \hline
        rank & 452,479,361 & 364,830,686 & 234,930,915 & 118,173,491 & 44,915,282 \\ 
        \hline
        \hline
        cohom. degree & $-9$ & $-10$ & $-11$ & $-12$ & $-13$ \\
        \hline
        rank & 12,284,967 & 2,239,008 & 236,445 & 10,350 & 0 \\
        \hline 
      \end{tabular}
    \end{center}
    \caption{Ranks in the raw derived pushforward cocomplex from 
    Example \ref{exp:2x3Example} for $k = 5$ and $l = 3$ using {\v C}ech cohomology}
  \end{table}
  However, the matrices for the coboundary maps of this cocomplex are rather sparse and 
  contain lots of 
  units. Therefore, this cocomplex can be simplified massively to a smaller one, 
  homotopy-equivalent to its original version;
  but the simplification itself simply takes too long: This particular example 
  ran on a reasonably sized compute server for about a month without 
  any substantial progress being made. 
  Given the size of the original mathematical problem 
  -- the restriction of $f$ to $X \subset \CC^{2\times 3}$ --
  which is rather small, 
  it seems fairly unacceptable that practical computation fails here. 
\end{example}

\section{Approximation via spectral sequences}
\label{sec:ApproximationViaSpectralSequences}

As a first approximation to the direct image, we may attempt to compute 
the \textit{spectral sequence} for the double complex 
\begin{equation}
  \label{eqn:TheDirectLimitDoubleComplex}
  \check C^{p, q}(\widetilde M^\bullet; \mathfrak U) \cong 
  \lim_{k\to \infty}
  \Hom\left(\Kosz^*(h_1^{k},\dots,h_N^{k}), M^\bullet\right)_0.
\end{equation}
On its first page, it has the entries
\[
  E_1^{p, q} = \bigoplus_{i_p = 1}^{b_p} R^q \pi_* \widetilde{S(-d_{p, i_p})}
\]
which can be given as the direct sum of the $q$-th homology of the complex 
\[
  \Hom\left(K_{\leq e_0}^\bullet(S), S(-d_{p, i_p})\right)
\]
with the minimal exponent vector $e_0$ for 
the degree $-d_{p, i_p}$.
The adaptation of these exponent vectors for the individual degrees 
$-d_{p, i_p}$ is crucial to the feasibility of the computations. 

\begin{example}
  \label{exp:2x3ExampleSpectralSequence}
  On the first page of the spectral sequence from Example \ref{exp:2x3Example}
  we find free $R$-modules $R^q\pi_* \widetilde M^p$ of ranks
  \[
    \begin{array}{c|ccccccccccccccc}
      q \backslash p & 0 & -1 & -2 & -3 & -4 & -5 & -6 & -7 & -8 & -9 & -10 & -11 & -12 & -13 & -14\\
      \hline
      0 & 1& 4&  5&   2&   0&    0&    0&     0&     0&     0&     0&     0&    0&   0& 0\\
      1 & 0& 3& 34& 127& 230&  230&  132&    42&     6&     0&     0&     0&    0&   0& 0\\
      2 & 0& 0&  6&  31&  62&   59&   26&     4&     0&     0&     0&     0&    0&   0& 0\\
      3 & 0& 0&  0&   0& 211& 1901& 7890& 19769& 32742& 36848& 27938& 13668& 3897& 492& 0
    \end{array}
  \]
  Even though the numbers in the bottom row become rather large, 
  this is still smaller than the ranks encountered in the truncated 
  total complex from Example \ref{exp:2x3Example} by a factor of more than 
  $10.000$. Moreover, the underlying complex of graded $S$-modules is the 
  resolution for a complex with non-trivial entries only in degrees $0$ up to $-3$ 
  and therefore cohomology can only appear up to the diagonal from $(0, -3)$ 
  down to $(3, -6)$ on the $\infty$-page of the spectral sequence. 
\end{example}

The first page of the spectral sequence is the natural habitat for everything 
that follows in the later pages since every entry $E_{k+1}^{p, q}$ is the homology of the 
differentials on the previous page:
\[
  E_{k}^{p-k, q+k-1}
  \overset{d}{\longrightarrow}
  E_{k}^{p, q}
  \overset{d}{\longrightarrow}
  E_{k}^{p+k, q-k+1}.
\]
For the construction of these differentials we may exploit the massive intrinsic block 
structure of the constituents of the double complex. This yields the angle 
of attack for the algorithm described here and the key to the major speedup 
compared to a direct application of Theorem \ref{thm:CechCohomologyAsDirectLimit}.
\begin{enumerate}[1)]
  \item As a first step, we may replace any right-bounded cocomplex 
    $(M^\bullet, \varphi^\bullet)$
    of finitely generated $S$-modules by an appropriate Cartan-Eilenberg resolution;
    thus we may, without loss of generality, assume every entry 
    $M^p \cong \bigoplus_{i_p =1}^{b_p} S(-d_{p, i_p})$ to be a direct sum of 
    free modules. 
  \item The column complexes of (\ref{eqn:TheDirectLimitDoubleComplex}) are 
    direct sums of direct limits of complexes of the form
    \[
      \lim_{e\to \infty} \Hom\left(K_{\leq e}^\bullet(S), S(-d)\right)_0
      = 
      \lim_{e\to \infty} \Hom\left(K_{\leq e}^\bullet(S), S\right)_{-d}
    \]
    which are strands of one and the same direct limit of complexes of 
    $S$-modules.
  \item For any fixed values for $e$ and $-d$, the matrices representing 
    the \v Cech differentials $\check \partial$ in the standard monomial bases 
    for the strands take a very particular and rather simple form: They are very sparse 
    and their non-zero entries are but $1$ and $-1$. These then form diagonal blocks 
    for the differentials of the full direct sums for the columns. In particular, 
    computations of kernels, images, and lifting is rather cheap along these maps.
  \item The matrices for all of the horizontal maps of the double complex induced 
    by the original 
    differentials $\varphi^p \colon M^p \to M^{p+1}$ again have a natural block 
    structure for the given decompositions of the respective domain and codomain. 
    Each one of these blocks is then given by the multiplication of a single 
    homogeneous polynomial $f$ on appropriate strands.
\end{enumerate}

To use this massive block structure to our advantage, we introduce 
a \textit{context object} \verb|ctx| associated to the multigraded 
ring $S$ which, on request, produces and caches the entries and maps 
of the complexes 
\[
  \Hom\left(K_{\leq e}^\bullet, S\right)
\]
for arbitrary exponent vectors $e = (e_1,\dots, e_N) \in \NN_0^N$. 
This is done together 
with their inclusion maps in the above direct limit and the induced maps on 
its strands. In particular, we store a monomial basis for every strand requested 
which then allows for computation of induced maps in constant time. 

Then, the assembly of the entries and the construction of the maps on the various pages of 
the spectral sequence is made from these constituent blocks. As a consequence, 
unavoidable Gr\"obner basis computations, e.g. for the involved kernels, only have 
to be made for $R$-modules which are subquotients of the entries on the first page, and 
the maps between them. The latter are even widely independent of one another and 
thus exhibit potential for distribution over different workers, i.e. an application 
of distributed computing. 

\section{The algorithm}
\label{sec:TheAlgorithm}

Computing spectral sequences of double complexes is a well-known concept; 
see e.g. \cite{Weibel94}. The upshot of this note is to systematically exploit 
the additional inherent structure of the \v Cech double complex to speed up this 
procedure. 

\subsection{The context object for the graded ring}

The key to speeding up the computation of the spectral sequence of the double 
complex is to disect the procedure into its many repetetive blocks, 
then precompute and cache these, and use them in a pipeline for resassembly 
of the result. This is facilitated via a context object associated 
to the homogeneous coordinate ring of $\PP\overset{\pi}{\longrightarrow} \Spec R$. \\

\verb|context_object(GradedPolyRing S) -> SpecSeqCtx|\\

\noindent
\textbf{Input:} A multigraded polynomial ring 
\[
  S = R[x_{1, 0}, \dots, x_{1, r_1}, x_{2, 0}, \dots, x_{m, r_m}], \quad 
  \deg(x_{i, k}) = (0, \dots, 0, \underbrace{1}_{i\textnormal{-th entry}}, 0, \dots, 0) \in \ZZ^m, k = 1,\dots, r_i.
\]
over a \textit{computable} ring $R$, representing a product of projective 
spaces $\mathbb P = \PP^{r_1} \times \dots \times \PP^{r_m} \overset{\pi}{\longrightarrow} \Spec R$.\\

\noindent
\textbf{Output:} A context object \verb|ctx| with the following associated methods 
for production and caching of certain objects.

\begin{itemize}
  \item A \textit{minimal cohomology model} 
    for the sheaf $\widetilde{S(d)}$ for every degree $d \in \ZZ^m$. 
    \begin{verbatim}
      minimal_cohomology_model(SpecSeqCtx ctx, Vector d) -> Cocomplex
    \end{verbatim}
    which returns a cocomplex of free $R$-modules with zero differentials 
    representing $R\pi_*\left(\widetilde{S(d)}\right)$.

  \item A \textit{minimal truncated \v Cech complex}
    for the sheaf $\widetilde{S(d)}$ for every degree $d \in \ZZ^m$. 
    \begin{verbatim}
      minimal_truncated_cech_complex(SpecSeqCtx ctx, Vector d) -> Cocomplex
    \end{verbatim}
    which returns the truncated a cocomplex 
    $\Hom\left(K_{\leq e_0(d)}^\bullet(S), S\right)_d$ for the 
    minimal exponent vector $e_0(d)$ of the respective degree.
    The \verb|minimal_cohomology_model| is a simplification of this cocomplex 
    up to homotopy.
    
  \item The inclusion and projection maps for the \verb|minimal_cohomology_model|
    into (resp. from) the \verb|minimal_truncated_cech_complex|.
    \begin{verbatim}
      cohomology_model_inclusion(SpecSeqCtx ctx, Vector d) -> CocomplexMorphism
      cohomology_model_projection(SpecSeqCtx ctx, Vector d) -> CocomplexMorphism
    \end{verbatim}
    Note that the particularly simple form of the cohomologies 
    $R^k\pi_*\left(\widetilde{S(d)}\right)$ allows us to write the latter as direct 
    summands of the cochain modules in the \verb|minimal_truncated_cech_complex|es 
    and to specify such maps. 
  
  \item The \textit{truncated \v Cech complex}
    for the sheaf $\widetilde{S(d)}$ for a degree $d \in \ZZ^m$ 
    and an exponent vector $e \geq e_0(d)$
    \begin{verbatim}
      truncated_cech_complex(SpecSeqCtx ctx, Vector d, Vector e) -> Cocomplex
    \end{verbatim}
    which returns the truncated a cocomplex 
    $\Hom\left(K_{\leq e}^\bullet(S), S\right)_d$.

  \item The inclusion and projection maps for the \v Cech complexes
    \begin{verbatim}
      cech_complex_map(SpecSeqCtx ctx, Vector d, Vector a, Vector b) 
                                                   -> CocomplexMorphism
    \end{verbatim}
    For $a \leq b$ this returns the inclusion map 
    \[
      \Hom\left(K_{\leq a}^\bullet, S\right)_d \hookrightarrow
      \Hom\left(K_{\leq b}^\bullet, S\right)_d
    \]
    and for $a \geq b$ the projection map. In case neither $a \leq b$ 
    nor $a \geq b$ can be decided, we form the supremum $c = \sup\{a, b\}$
    with entries 
    \[
      c_i = \max\{a_i, b_i\}
    \]
    for $i = 1, \dots, m$ and return the composed map for the pairs $(a, c)$ 
    and $(c, b)$.

  \item The morphism 
    induced by a multiplication 
    \[
      f \cdot - \colon S(d) \to S(d-\deg(f)), \quad g \mapsto f\cdot g
    \]
    with a homogeneous polynomial $f \in S$ on the truncated \v Cech complex 
    for an exponent vector $e \in \NN_0^m$.
    \begin{verbatim}
      cech_complex_map(SpecSeqCtx ctx, Vector d, Vector e, Poly f) 
                                                  -> CocomplexMorphism
    \end{verbatim}
    It turns out to be more efficient to not actually cache such maps, but 
    only the identifications of the generators of the strands with the monomials 
    in the complexes $\Hom\left(K_{\leq e}^\bullet, S\right)$, e.g. via 
    dictionaries. Then induced maps for $f$ can be assembled in constant time 
    from that data.
\end{itemize}

\subsection{The objects for the spectral sequence and its pages}
Computation of spectral sequences should be implemented in a lazy way as users 
tend to be interested only in specific entries of certain pages. Moreover, 
a given spectral sequence may not converge uniformly, meaning that there might 
be no specific page number $k$ such that $E_k^{i, j} \cong E_\infty^{i, j}$ 
is isomorphic to the $\infty$-page for all entries $(i, j)$ at once. 
For these reasons, creating a spectral sequence will therefore only initialize 
an object with associated methods for the production of the pages. \\

\begin{verbatim}
  spectral_sequence(GradedPolyRing S, Cocomplex M) -> CohomSpecSeq
\end{verbatim}

\noindent
\textbf{Input:} A graded ring $S$ as above and a right-bounded cocomplex 
of graded $S$-modules $M$ with homogeneous differentials of degree zero.\\

\noindent
\textbf{Output:} A \verb|CohomSpecSeq| object \verb|css| with the following 
methods associated to it for further computations 

\begin{itemize}
  \item Production and caching of the $k$-th page, $k \geq 1$:
    \begin{verbatim}
      page(CohomSpecSeq css, Int k) -> SpecSeqPage
    \end{verbatim}
  \item Return the cached context object for the graded ring $S$:
    \begin{verbatim}
      context_object(CohomSpecSeq css) -> SpecSeqCtx
    \end{verbatim}
  \item Getter functions for the graded ring $S$ via \verb|graded_ring| 
    and the complex of graded modules $M$ via \verb|graded_complex|.
\end{itemize}

\medskip
For an individual page \verb|cssp| of type \verb|SpecSeqPage| returned by 
the function \verb|page(css, k)| we provide the following functionality.
\begin{itemize}
  \item Production and caching of the entries $E_k^{i, j}$
    \begin{verbatim}
      entry(SpecSeqPage cssp, Int i, Int j) -> Module
    \end{verbatim}
    as $R$-modules. 
  \item Production and caching of the maps $\D_k \colon E_k^{i, j} \to E_k^{i+k, j-k+1}$
    \begin{verbatim}
      map(SpecSeqPage cssp, Int i, Int j) -> ModuleMorphism
    \end{verbatim}
    as morphisms of $R$-modules
  \item A getter for the number $k$ of the page via \verb|page_number(cssp)| 
    and for the spectral sequence \verb|css| of which this is the $k$-th page 
    via \verb|spectral_sequence(cssp)|.
\end{itemize}

\subsection{Computation of the first page}
\label{sec:ComputationsFirstPage}

Deviating from the common literature, we start the computation of the 
spectral sequence on page $k=1$. As the assembly of this initial page 
is somewhat different from its successors, it is deserves its own description.

\begin{verbatim}
  function produce_entry_on_first_page(SpecSeqPage cssp, Int i, Int j)
    css <- spectral_sequence(cssp)
    C <- graded_complex(css)
    ctx <- context_object(css)
    return direct_sum([minimal_cohomology_model(ctx, degree(g))[j] 
                        for g in generators(C[i])])
  end
\end{verbatim}
Here we assume that the $i$-th cochain module of a \verb|Cocomplex| $C^\bullet$
can be retrieved via \verb|C[i]|. Moreover, the given implementation of the 
category of $R$-modules must provide a method for building direct sums 
so that the output $M$ comes equipped with functionality for 
\begin{verbatim}
  canonical_injection(Module M, Int k) -> ModuleMorphism
  canonical_projection(Module M, Int k) -> ModuleMorphism
\end{verbatim}
for the respective $k$-th summand.

\medskip
The production of the maps on the first page is straightforward, but technically 
tedious to write down. The relevant map in the original graded complex $C^\bullet$ 
is of the form 
\[
  \bigoplus_{k=0}^{b_i} 
  S(-d_{i, k}) \to 
  \bigoplus_{l=0}^{b_{i+1}} 
  S(-d_{i+1, l})
\]
and has a representing matrix $A \in S^{b_i \times b_{i+1}}$ with 
homogeneous entries $f = A_{k, l}$ of degree $\deg f = -d_{i+1, l} + d_{i, k}$.
Every such non-zero entry provides a block for the induced map on the first page 
which arises as the composition of the following maps:
\[
  \begin{xy}
    \xymatrix{
      E_1^{i, j} \cong 
      \bigoplus_{k = 1}^{b_{i}}
      R^j\pi_*\left(\widetilde{S(-d_{i, k})}\right) \ar[r]^{\textnormal{pr}_k} &
      R^j\pi_*\left(\widetilde{S(-d_{i, k})}\right) \ar@{^{(}->}[r] & 
      \Hom\left(K_{\leq e_0(-d_{i, k})}^\bullet, S\right)_{-d_{i, k}} 
      \ar[dll]^{f\cdot -} \\
      \Hom\left(K_{\leq e_0(-d_{i, k})}^\bullet, S\right)_{-d_{i+1, l}} 
      \ar[r] & 
      \Hom\left(K_{\leq e_0(-d_{i+1, l})}^\bullet, S\right)_{-d_{i+1, l}}
      \ar@{^{(}->}[r]& 
      R^j\pi_*\left(\widetilde{S(-d_{i+1, l})}\right) \ar@{^{(}->}[dll]^{\textnormal{inc}_l}\\
      E_1^{i+1, j} \cong
      \bigoplus_{l = 1}^{b_{i-1}}
      R^j\pi_*\left(\widetilde{S(-d_{i+1, l})}\right). & & 
    }
  \end{xy}
\]
where $\textnormal{pr}_k$ and $\textnormal{inc}_l$ denote the projection and the 
inclusion of the $k$-th (resp. $l$-th) direct summand. 
The decomposition into these blocks and the reassembly of the result 
is spelled out in the somewhat convoluted pseudo code below. The main 
point is that from the cohomology representations on the first page 
we directly dive into the realm of cached objects from the context object 
\verb|ctx| associated to the graded ring $S$ and we avoid building any of 
the (truncations of) the \v Cech complexes for the entries of $M^\bullet$ and 
the induced maps there.
\begin{verbatim}
  function produce_map_on_first_page(SpecSeqPage cssp, Int i, Int j)
    css <- spectral_sequence(cssp)
    M <- graded_complex(css)
    ctx <- context_object(css)
    phi <- map(M, i) # the original map of graded modules
    dom <- entry(cssp, i, j) # the domain of the map to be produced
    cod <- entry(cssp, i+1, j) # the codomain of the map to be produced
    img_gens <- [] # initialize an empty list for the images of the generators
    for g0 in generators(dom) # iterate through the generators of the domain
      img_gen <- zero(cod) # initialize a variable for its image
      for k in 1 to number_of_generators(M[i])
        d0 <- -degree(generator(M[i], k))
        e0 <- minimal_exponent_vector(ctx, d0)
        pr <- canonical_projection(dom, k)
        g1 <- pr(g0) # the k-th block of this module element
        inc <- cohomology_model_inclusion(ctx, d0) 
        g2 <- inc(g1) # its image in the `minimal_truncated_cech_complex`
        orig_img_gen <- phi(generator(M[i], k)) # the image of the k-th 
                                                # generator of M[i]
        for l in 1 to number_of_generators(M[i+1])
          f <- orig_img_gen[l] # the l-th component of that generator
          d1 <- -degree(generator(M[i+1]), l) # equal to d0 + degree(f)
          mult_f <- cech_complex_map(ctx, d0, e0, f) # the induced map
          g3 <- mult_f(g2) # the image for this block of the vector
          e1 <- minimal_exponent_vector(ctx, d1)
          pr <- cech_complex_map(ctx, d1, e0, e1) 
          g4 <- pr(g3)
          pr <- cohomology_model_projection(ctx, d1)
          g5 <- pr(g4)
          inc <- canonical_inclusion(cod, l)
          g6 <- inc(g5)
          img_gen <- img_gen + g6
        end
        push(img_gens, img_gen) # push the result for g0 to the list 
                                # of images of the generators
      end
    end
    return hom(dom, cod, img_gens) # finally create the morphism
  end
\end{verbatim}
Note that this pseudo code is not yet an efficient implementation; for instance, 
the various loops can be enhanced with early abort statements.
As such improvements are straightforward to make in 
practice, but would clutter up the above conceptual pseudo code, we 
deliberately refrain from spelling them out. The interested reader is 
referred to the actual implementation in \verb|OSCAR| for details.

\subsection{Computations of higher pages}

The entry $E_{p+1}^{i, j}$ on the $p+1$-st page is computed as the cohomology 
of the maps on the previous page 
\[
  E_{p}^{i-p, j+p-1}
  \overset{d_p}{\longrightarrow}
  E_{p}^{i, j}
  \overset{d_p}{\longrightarrow}
  E_{p}^{i+p, j-p+1}, 
\]
i.e. $E_{p+1}^{i, j} \cong \ker d_p / \im d_p$.
Let us, for the sake of brevity, simply write 
\[
  \check C^{i, j}
  = 
  \lim_{e \to \infty} 
  \bigoplus_{k=1}^{b_i} 
  \Hom\left(K_{\leq e}^j, S\right)_{-d_{i, k}}
\]
for the $(i, j)$-th entry of the underlying \v Cech double 
complex $\check C^\bullet(M^\bullet; \mathfrak U)$. 
Then the differentials $d_p$ can be lifted to morphisms 
\[
  \check \psi_p \colon \check Z_{p-1}^{i, j} \to \check C^{i+p, j-p+1}
\]
as depicted in the following diagram:
\[
  \begin{xy}
    \xymatrix{
      \check Z_0^{i, j} \ar@{^{(}->}[r] 
      \ar@/_1pc/[rr]_{\check\psi_1}&
      \check C^{i, j} \ar[r]^\varphi 
      &
      \check C^{i+1, j} & & \\
      \check Z_1^{i, j} \ar@{^{(}->}[u] 
      \ar@/_1pc/[rrr]_{\check\psi_2}&
      & 
      \check C^{i+1, j-1} \ar[u]^{\check \partial} \ar[r]^\varphi 
      & 
      \check C^{i+2, j-1}
      & \\
      \check Z_2^{i, j} \ar@{^{(}->}[u] 
      \ar@/_1pc/[rrrr]_{\check\psi_3}&
      & & \check C^{i+2, j-2} \ar[u]^{\check \partial} \ar[r]^\varphi& 
      \check C^{i+3, j-2} \\
      \vdots \ar@{^{(}->}[u] 
      & & & & \ddots \\
      \check Z_{p-1}^{i, j} \ar@{^{(}->}[u] 
      \ar@/_1pc/[rrrrr]_{\check\psi_p}&
      & & & \check C^{i+p-1, j-p+1} \ar[r]^\varphi & 
      \check C^{i+p, j-p+1}
      \\
      \check Z_{p}^{i, j} \ar@{^{(}->}[u] 
      \ar@/_1pc/[rrrrrr]_{\check\psi_{p+1}}&
      & & & & \check C^{i+p, j-p} \ar[u]^{\check \partial} \ar[r]^\varphi & 
      \check C^{i+p+1, j-p} \\
    }
  \end{xy}
\]
The maps $\check \psi_p$ can be constructed inductively, 
starting from $\check \psi_0 = \check \partial$.
As $\check\psi_1$ we then take the restriction of $\varphi$ to the 
kernel $\check Z_0^{i, j}$ of $\check \partial$. 
Assuming $\check \psi_p$ to already be established, we proceed with 
the construction of $\check \psi_{p+1}$ as follows.
\begin{enumerate}[1)]
  \item Let 
    \begin{eqnarray*}
      \check B_{p-1}^{i+p, j-p+1} &=& \quad
      \im\left( \check \psi_{p-1}' \colon 
        \check Z_{p-2}^{i+1, j-1} \to \check C^{i+p, j-p+1}
      \right)\\
      & & + \quad
      \im\left( \check \psi_{p-2}' \colon 
        \check Z_{p-3}^{i+2, j-2} \to \check C^{i+p, j-p+1}
      \right)\\
      & & + \quad
      \quad
      \cdots \\
      & & + \quad
      \im\left( \check \psi_{1}' \colon 
        \check Z_{0}^{i+p-1, j-p+1} \to \check C^{i+p, j-p+1}
      \right)\\
      & & + \quad
      \im\left( \check \psi_{0}' = \check \partial \colon 
        \check C^{i+p, j-p} \to \check C^{i+p, j-p+1}
      \right)
    \end{eqnarray*}
    be the aggregated image of the previous incoming maps 
    $\check \psi_q'$ to $\check C^{i+p, j-p+1}$ and set 
    \[
      \check Z_p^{i, j} = \check \psi_p^{-1}\left(
      \check B_{p-1}^{i+p, j-p+1}
      \right).
    \]
  \item Compute a generating set 
    $u_1,\dots, u_N\in \check Z_p^{i, j} \subset \check C^{i, j}$ for this module.
  \item For every generator $u$, write 
    \[
      v = \check \psi_p(u) = 
      \check \psi'_{p-1}(u'_1) + 
      \check \psi'_{p-2}(u'_2) + 
      \cdots + 
      \check \psi'_{1}(u'_{p-1}) + 
      \check \psi'_{0}(u'_p)
    \]
    for some choices of $u'_k$ and, in particular, of $u'_p \in \check C^{i+p, j-p}$.
  \item Map $u'_p$ to $w := \varphi(u'_p) \in \check C^{i+p+1, j-p}$;
    this element $w$ then constitutes the image of the chosen generators $u$ 
    of $\check Z_p^{i,j}$ under $\check \psi_{p+1}$. Repeating steps 3) and 4) for every 
    such generator provides us with the full map $\check\psi_{p+1}$.
\end{enumerate}

\noindent That this is a correct procedure to produce the maps 
\begin{eqnarray*}
  d_p \colon 
  E_p^{i, j} \cong \check Z_{p-1}^{i, j}/ \check B_{p-1}^{i, j}
  & \longrightarrow &
  E_p^{i+p, j-p+1} \cong \check Z_{p-1}^{i+p, j-p+1}/ \check B_{p-1}^{i+p, j-p+1},
  \\
  a + \check B_{p-1}^{i, j} & \mapsto & \check\psi_p(a) + \check B_{p-1}^{i+p, j-p+1}
\end{eqnarray*}
on the respective pages of the spectral sequence for $p\geq 2$ 
is a somewhat tedious exercise 
which we leave to the reader. It can be carried out with the help of any 
textbook on homological algebra; e.g. \cite{Weibel94}. 

\medskip

The challenge we are facing at this point is that the modules 
\[
  \{0\} 
  \subset 
  \check B_0^{i, j} 
  \subset 
  \check B_1^{i, j} \subset \dots \subset 
  \check B_p^{i, j} 
  \subset \check Z_p^{i, j} \subset \dots \subset \check Z_1^{i, j} \subset 
  \check Z_0^{i, j} \subset 
  \check C^{i, j}
\]
can not be assumed to be finitely generated over $R$. 
However, already the first quotient 
\[
  E_1^{i, j} = \check Z_0^{i, j} / \check B_0^{i, j} \cong A^r
\]
is isomorphic to a \textit{finite} free $A$-module of some rank $r$. 
This follows from Theorem 
\ref{thm:CechCohomologyAsDirectLimit}, applied to the particular 
case of a complex with only one single module, in conjunction with 
Lemma \ref{lem:HomotoyEquivalence}.
Such finite modules form a much more preferable environment for the computations 
of the spectral sequence and the following lemma describes the key 
observation as to why we may actually restrict our attention 
to these.

\begin{lemma}
  \label{lem:KernelSplitting}
  Let $M = M^p = \bigoplus_{i_p=1}^{b_p} S(-d_{p, i_p})$ be the single free 
  module of $M^\bullet$ at cohomological degree $p$ and consider its \v Cech 
  complex
  \[
    \cdots
    \overset{\check \partial}{\longrightarrow}
    \lim_{e\to \infty} \Hom\left(K_{\leq e}^{j-1}, M\right)_0 
    \overset{\check \partial}{\longrightarrow}
    \lim_{e\to \infty} \Hom\left(K_{\leq e}^j, M\right)_0
    \overset{\check \partial}{\longrightarrow}
    \lim_{e\to \infty} \Hom\left(K_{\leq e}^{j+1}, M\right)_0 
    \overset{\check \partial}{\longrightarrow}
    \cdots.
  \]
  Then 
  \begin{equation}
    \label{eqn:KernelSplitting}
    \ker \check \partial \cong 
      R^j\pi_*\left(\widetilde{M}\right)
    \oplus
    \im(\check \partial)
  \end{equation}
  splits as a direct sum of free $R$-modules with the first summand being finite 
  and contained in $\Hom\left(K_{\leq e}^j, M\right)_0$ with $e$ 
  the supremum of the the minimal exponent vectors $e_0(-d)$ for $-d$ 
  ranging through the shifts $-d_{p, i_p}$ for $M$.
\end{lemma}

\begin{proof}
  This follows directly from the fact that direct images commute with 
  direct sums and the explicit construction of $K_{\leq e}^\bullet$
  in Section \ref{sec:ApproximationByFiniteSubmodules}; in particular 
  Lemma \ref{lem:HomotoyEquivalence} and the considerations on monomial 
  diagrams.
\end{proof}

Let 
\[
  \{0\} = B_0^{i, j} \subset B_1^{i, j} \subset \dots \subset 
  B_p^{i, j} 
  \subset Z_p^{i, j} \subset \dots 
  \subset Z_2^{i, j}
  \subset Z_1^{i, j} \subset 
  E_1^{i, j} \cong R^r
\]
be the images of the modules $\check B_q^{i, j}$ and $\check Z_q^{i, j}$ 
in $E_1^{i, j} \cong R^j\pi_*(M^i) \cong \check Z_0^{i, j}/ \check B_0^{i, j}$. Using Lemma 
\ref{lem:KernelSplitting} it is easy to see that then 
\begin{equation}
  \label{eqn:DirectSumDecompositionBandZ}
  \check B_q^{i, j} \cong B_q^{i, j} \oplus \im\check\partial
  \qquad \textnormal{and} \qquad
  \check Z_q^{i, j} \cong Z_q^{i, j} \oplus \im\check\partial.
\end{equation}
Let 
\[
  \psi_p \colon Z_{p-1}^{i, j} \to \check C^{i+p, j-p+1}
\]
be the restriction of $\check \psi_p$ to the first summand of $\check Z_{p-1}^{i, j}$ 
in that decomposition. 

\begin{lemma}
  \label{lem:RestrictedPsi}
  The map $d_p \colon E_p^{i, j} \to E_p^{i+p, j-p+1}$ induced by 
  $\check \psi_p$ depends only on $\psi_p$.
\end{lemma}

\begin{proof}
  This trivial observation follows from the definition of the domain 
  of $d_p$: The modulus of $E_p^{i, j}$ comprises the summand 
  $\im \check\partial$ on which any two candidates for 
  $\check \psi_p$ may differ if they were to give the same $\psi_p$. 
\end{proof}

For the computation of the higher pages of the spectral sequence, 
we may therefore proceed with constructing only the maps $\psi_p$. 
Even though their codomains are only direct limits of finitely generated 
$R$-modules, the finiteness of their domains over $R$ allows us to describe 
any such map in finite terms. 
To achieve this, we enrich the data stored in a \verb|SpecSeqPage| 
by 
\begin{verbatim}

  lifted_kernel_generators(SpecSeqPage cssp, Int i, Int j) -> List

\end{verbatim}
On the $p+1$-st page \verb|cssp| (standing for $E_{p+1}$) and an index pair $(i, j)$ 
this computes and stores the images of the generators 
$u_1,\dots, u_n$ of $Z_p^{i-p-1, j+p}$
for the incoming map  
under $\psi_{p+1}$ as a list of elements $\left\{w_k\right\}_{k=1}^n$ 
in $\check C^{i, j}$.

Recall that 
\[
  \check C^{i, j} = 
  \lim_{e \to \infty} 
  \bigoplus_{k=1}^{b_i} 
  \Hom\left(K_{\leq e}^\bullet, S\right)_{-d_{i, k}}
\]
is a limits of direct sums. Storing an element $w$ in this object is therefore 
equivalent to storing an exponent vector $e$ and the components 
of $w$ in the modules $\Hom\left(K_{\leq e}^\bullet, S\right)_{-d_{i, k}}$.
The $k$-th entry of the \verb|List| returned by \verb|lifted_kernel_generators| 
therefore consists of precisely this datum: An integer vector $e$ and a 
list of pairs $(l, w_{k, l})$ where $l$ is an integer for the index of 
the block in the direct sum and $w_{k, l}$ is the respective non-zero summand. 
Thus, we store the images of the generators of $\psi_{p+1}$ in a 
\textit{sparse block format}.\\

\medskip
To compute $\psi_{p+1}$ from $\psi_p$ efficiently along the above lines, 
we proceed as follows.
\begin{enumerate}[1)]
  \item Given the \verb|lifted_kernel_generators| for $\psi_p$, 
    we compute the map 
    \[
      \overline \psi_p \colon Z_{p-1}^{i, j} \to E_p^{i+p, j-p+1} 
      \cong Z_p^{i+p, j-p+1}/B_p^{i+1, j-p+1}
    \]
    as an honest map of finitely generated $R$-modules.
    This can be done by mapping the list $\{(l, w_{k, l})\}$ for each 
    image element $w_k$ for $\psi_p$ to $E_p^{i+p, j-p+1}$ in a blockwise 
    manner, similar to the procedure described for the first page in Section 
    \ref{sec:ComputationsFirstPage}. By construction, the kernel 
    of $\overline \psi_p$ coincides with $Z_p^{i, j}$.
  \item Computing a set of generators for $Z_p^{i, j}$ 
    is carried out exclusively in the category of finitely generated $R$-modules 
    which we treat as a black box for the description of this algorithm. 
    Note again that our preparations allow us to keep the $R$-modules 
    involved in this step relatively small. 
  \item Given a generator $u$ of $Z_p^{i, j}$, we compute its image 
    under $\psi_p$ in a sparse block form as follows. Let 
    $\lambda = (\lambda_1, \dots, \lambda_n)$ be the coefficients of a 
    linear combination for 
    \[
      u = \sum_{k=1}^n \lambda_k \cdot u_k'
    \]
    in the generators $\{u_k'\}_{k=1}^n$ of $Z_{p-1}^{i, j}$. 
    Forming linear combinations 
    commutes with decomposition into blocks of direct summands and, hence, the 
    image of $u$ can be given by a list of pairs $(l, v_l)$ with $l$ the index 
    of the block and 
    \[
      v_l = 
      \sum_{k=1}^n
      \lambda_k \cdot w'_{k, l}
    \]
    for $w'_{k, l}$ the $l$-th summand of the image $w'_k$ of the $k$-th generator 
    $u'_k$ of $Z_{p-1}^{i, j}$ under $\psi_p$. 
    Note that, in order to form this linear combination, 
    we will -- in the general case -- 
    need to promote all summands to a common denominator, i.e. work 
    with the supremum of the exponent vectors stored for the $w'_k$ and $u$. 
    Again, this can be done using the functions provided by the 
    \verb|SpectralSequenceCtx| object for the underlying graded ring. 

    Once we have computed $v = \{(l, v_l)\}_l$ in its sparse block format, 
    we may write 
    \[
      v = \psi_p(u) = 
      \psi'_{p-1}(u'_1) + 
      \psi'_{p-2}(u'_2) + 
      \cdots + 
      \psi'_{1}(u'_{p-1}) + 
      \check \partial(u'_p)
    \]
    for the incoming maps $\psi'_q$ as in step 3) above. 
    In practice, we first check for direct liftability of $v$ along $\check \partial$. 
    If that fails, the obstruction to lifting $v$ is actually contained in 
    $Z_0^{i+p, j-p+1}$, i.e. the first summand in (\ref{eqn:KernelSplitting}). 
    Let $\overline v$ be the projection of $v$ to that summand.
    Since we know the images $\{v'_l\}_{l=1}^N$ 
    of the generators for the incoming maps $\psi'_q$ 
    and their projections $\overline v'_l$ to $Z_0^{i+p, j-p+1}$, we write 
    \[
      \overline v = \sum_{l=1}^N \mu_l \cdot \overline v'_l,
    \]
    a task to be carried out in the category of \textit{finite} $R$-modules.
    Replacing $v$ by 
    \[
      v - \sum_{l=1}^N \mu_l \cdot v'_l,
    \]
    -- again performing the linear combination blockwise -- we then get an 
    element which certainly lifts along the (block-diagonal) map 
    $\check \partial$ to some preimage $u_p' \in \check C^{i+p, j-p}$.

  \item Mapping $u_p'$ to $\check C^{i+p+1, j-p}$ 
    along $\varphi$ is done similarly to the procedures discussed in 
    Section \ref{sec:ComputationsFirstPage}. 
\end{enumerate}

\noindent
As mentioned earlier, the differential 
$d_{p+1} \colon E_{p+1}^{i, j} \to E_{p+1}^{i+p+1, j-p}$ 
is given by the restriction of 
\[
  \psi_{p+1} \colon Z_{p}^{i, j} \to \check C^{i+p+1, j-p}
\]
to $E_{p+1}^{i, j} = \check Z_p^{i, j}/\check B_p^{i, j} \cong Z_{p}^{i, j}/B_p^{i, j}$ 
and $E_{p+1}^{i+p+1, j-p}$. Therefore, this concludes the computation of the induced maps 
on the higher pages of the spectral sequence and the entries on the subsequent page.

\section{Timings}

We plan to explore timings in relevant examples with comparable implementations 
in Macaulay2 \cite{M2} -- namely the packages \cite{ToricHigherDirectImagesSource}, 
\cite{TateOnProductsSource}, and \cite{MultigradedBGGSource} -- in a forthcoming 
update for this preprint. 

For the existing implementation in OSCAR \cite{Oscar} we can compare runtimes 
between the above implementation for spectral sequences and the straightforward computation 
of the direct image via Theorem \ref{thm:Truncations} with subsequent simplification 
up to homotopy (``prune for complexes''). 
For \cite[Example 5.2]{Zach24} we get the timings listed in 
Table \ref{tab:InternalOscarTimings}
on a laptop with 32GB of RAM and a 12th Gen Intel Core i5-1245U processor. 
\begin{table}
  \label{tab:InternalOscarTimings}
  \begin{center}
    \begin{tabular}{|l|c|c|}
      \hline
      Task & Time & Allocated memory \\
      \hline
      computation of the spectral sequence in Section \ref{sec:ApproximationViaSpectralSequences} 
      & $8.9s$ & $799.9MB$ \\
      \hline
      direct image via truncated \v Cech complexes and simplification & $551.8s$ & $55.8GB$ \\
      \hline 
    \end{tabular}
  \end{center}
  \caption{Timings for \cite[Example 5.2]{Zach24} for different methods within Oscar}
\end{table}

As already mentioned before, the computations for \cite[Example 5.3]{Zach24} 
(Example \ref{exp:2x3Example} here) along Theorem \ref{thm:Truncations} do not finish 
within a month of computing time. The entries of the $\infty$-page of the 
spectral sequence, however, can be retrieved in less than two
hours on a compute server with $120$ GB of RAM, with none of the individual 
entries taking more than one hour to compute.

\section{Conclusion and outlook}

We have presented and explained a method to speed up computation of a spectral sequence 
converging to the higher direct images of a complex of coherent sheaves along the projection 
of a product of projective spaces $\PP$ over an arbitrary computable ring $R$ 
to its base $\Spec R$.
If we assume the input to be a right-bounded cocomplex of \textit{free} 
graded $S$-modules, $S$ being the homogeneous coordinate ring for $\PP$, 
then our algorithm reduces the problem directly to procedures on finitely 
generated $R$-modules. 

A comparison of runtimes with the direct approach via 
truncated \v Cech complexes, as described in Theorem \ref{thm:Truncations}, 
suggests a major increase in efficiency for bigger examples. 
However, this comparison is tentatively unfair, as the latter method 
is known to be quite expensive and the result is richer 
in mathematical information than just the terms on the $\infty$-page 
of the spectral sequence. Comparison of timings with other, more efficient 
methods for computing higher direct images will follow in an updated version 
of this preprint. Since the described algorithm uses only a minimum of information 
to directly reduce a massively redundant task to 
computations on $R$-modules, we believe that 
the theoretical efficiency of our approach should be close to optimal. 
Moreover, the independence of one another for most pairs of entries on the pages of 
the spectral sequence yields another angle of attack for speedup via parallelization 
of tasks. 

The algorithm and its implementation have been described for products of 
projective spaces. But this should not be the only setting in which our considerations 
are applicable. For instance, it would be interesting to explore, to which extent 
the methods can be generalized to complexes of graded modules on toric varieties 
and their Cox rings. Yet another direction we plan to pursue, 
is to actually recover a complex 
for the direct image up to quasi-isomorphism 
from the information gathered in the spectral sequence 
and the process of computing it. 

\section*{Acknowledgements}

This work has been funded by the Deutsche Forschungsgemeinschaft 
(DFG, German Research Foundation) – Project-ID 286237555 – TRR 195. 

\printbibliography
\end{document}